\documentclass[a4paper]{article}
\usepackage{amsthm,amssymb,amsmath,enumerate,graphicx,psfrag}
\usepackage{tikz}

\newtheorem{definition}{Definition}

\newtheorem{theorem}[definition]{Theorem}

\newtheorem{lemma}[definition]{Lemma}

\newcommand{\comment}[1]{}

\newcommand{\emtext}[1]{\text{\em #1}}

\newcommand{\sm}{\setminus}

\renewcommand{\epsilon}{\varepsilon}

\newcommand{\udg}{\textsc{udg}}

\newcommand{\conv}{\ensuremath{\text{\rm conv}}}
\newcommand{\dist}{\ensuremath{\text{\rm dist}}}

\newcounter{ccount}
\newcommand{\case}{\refstepcounter{ccount}\medskip\noindent{\em Case \arabic{ccount}:~}}

\title{Colouring stability two unit disk graphs}
\author{Henning Bruhn}
\date{}

\begin{document}
\maketitle

\begin{abstract}
We prove that every stability two unit disk graph 
has chromatic number at most $\frac{3}{2}$ times its clique number. 
\end{abstract}

\section{Introduction}

 A unit disk graph (or \udg\ for short) is defined on a point set in the plane, where 
two points are considered as adjacent vertices if their distance is 
at most one. As a very basic model for wireless devices, unit disk graphs
as well as more sophistacted variants have attracted quite a lot 
of interest, and there exists an extensive literature on their 
subject. One of the earliest application is due to Hale~\cite{Hale80}
who considered them in the context of the frequency assignment problem.
There, the task is to assign different frequencies to the wireless
devices in order for them to communicate with a base station without
interference. An edge between two points is understood as the devices
being close enough to interfere with each other, so that an
assigned frequency corresponds to a stable set in the graph. The 
frequency assignment problem then becomes a graph colouring problem,
and naturally, it is desired to assign as few frequencies as possible. 

In this article I will treat the colouring  of a
unit disk graph from a structural point of view. Unit disk graphs
distinguish themselves from general graphs in that their
chromatic number can be 
upper-bounded in terms of the clique number. 
The best known bound  is 
due to Peeters:

\begin{theorem}[Peeters~\cite{Peeters91}]
A unit disk graph $G$ can be coloured with at most $3\omega(G)-2$
colours.
\end{theorem}

How good is that bound? Malesi{\'n}ska, Piskorz and Wei\ss enfels~\cite{MPW98}
gave the following lower bound. 
Consider the class of graphs 
$C_n^k$ on vertex set $0,\ldots, n-1$
where $i$ and $j$ are adjacent if $|i-j|\leq k-1$. Observe that
we may realise any $C_n^k$ as a unit disk graph by placing 
the vertices at equidistance on a circle of appropriate radius. 
Moreover, we see that $C_{3k-1}^k$ has stability $\alpha=2$ and
clique number $\omega=k$,
from which we deduce for the chromatic number that
 $\chi(C_{3k-1}^k)\geq \frac{3k-1}{2}=\frac{3}{2}\omega-\frac{1}{2}$.

Clearly, between an upper bound of essentially $3\omega$ and a lower bound
of $\frac{3}{2}\omega$ there is quite a bit of scope for improvement. 
In this article, we will prove 
that the lower bound is the true one if the unit disk graph 
has, in addition, at most stability two:

\begin{theorem}\label{mainthm}
A unit disk graph $G$ with $\alpha(G)\leq 2$ can be coloured with at most 
$\frac{3}{2}\omega(G)$ colours.
\end{theorem}

The theorem shows that, at the very least, the example for the lower
bound cannot easily be improved. Moreover, I contend that it gives some evidence
for believing that the true bound is closer to $\frac{3}{2}\omega$ than to 
Peeters' bound. 

There is some more evidence for this belief. 
There are two classes of \udg s, for which it is known that their
chromatic number is bounded by $\frac{3}{2}\omega$. 
The first of these is the class of triangle-free
\udg s: A triangle-free \udg\
is planar, see Breu~\cite{BreuPHD}, and thus by Gr\"otzsch' theorem $3$-colourable~\cite{Gr59}.
The second class consists of augmentations of induced subgraphs of the triangular
lattice in the plane. McDiarmid and Reed~\cite{McDR00}
show that these can be coloured with at most $\frac{4\omega+1}{3}$
colours.
Another piece of evidence is provided by McDiarmid~\cite{McDiarmid03},
who investigates fairly general models of 
random unit disk graphs. In that context, it turns out
that with high probability the 
chromatic number is very close to the clique number, much closer even 
than the factor of $\frac{3}{2}$ of the lower bound. 
Finally, considering fractional instead of ordinary colourings,
Gerke and McDiarmid~\cite{GMcD01} prove that the fractional chromatic
number is bounded by
$2.2 \omega(G)$ for any unit disk graph $G$.

\medbreak

Optimisation problems in \udg s, and in particular, colouring \udg s algorithmically
have
attracted some attention. We just mention
Marathe et al.~\cite{MBHRR95} who give a $3$-approximation colouring algorithm
and the result by
Clark, Colbourn and Johnson~\cite{CCJ90} that $3$-colourability remains NP-complete for \udg s.
We refer to 
Balasundaram and Butenko~\cite{BB08} for a survey on several optimisation problems in \udg s.

\medbreak

The paper is organised as follows. After briefly stating some of the basic definitions 
that we are going to use we will proceed with the proof of our main result, Theorem~\ref{mainthm},
in Section~\ref{main}. The key lemma on which 
the proof of Theorem~\ref{mainthm} rests will be deferred to Section~\ref{keysec}. 
In Section~\ref{geom} we will discuss which geometric insights will be exploited.

\section{Definitions}
For general graph-theoretic notation and concepts we refer to Diestel~\cite{diestelBook10}.

 Let $G$ be 
a graph. A \emph{clique} of $G$ is a subgraph in which any two 
vertices are adjacent. A \emph{stable set} of $G$ is a subgraph
or vertex set so that no two vertices are adjacent. The size of 
the largest clique is denoted by $\omega(G)$, while the 
size of the largest stable set is $\alpha(G)$, the \emph{stability of $G$}.
We denote the chromatic number by $\chi(G)$, and define $\overline\chi(G)$,
the \emph{clique partition number}, 
to be the chromatic number of the completement of $G$. 
A vertex $v$ is \emph{complete to} some vertex set $X$ if every vertex in $X$
is adjacent to $v$. A vertex set $U$ is \emph{complete to} $X$ is every
vertex in $U$ is complete to $X$.

A \emph{unit disk} is a closed disk of radius~$1$ in the plane.
Unit disk graphs can be represented in two ways: In the intersection model 
the vertices are unit disks in the plane and two of them are adjacent
if and only if the disks intersect; and in the distance model, the vertex set is 
a point set in the plane, and  any two vertices 
are adjacent if and only if their distance is at most~$1$.
\emph{We work exclusively with the distance model.}

Moreover, we always see a unit disk graph as a concrete geometric object, that is, 
the vertex set is indeed a subset of points in $\mathbb R^2$. 
So, every vertex \emph{is} a point in the plane. As a consequence
we do not allow two vertices to be represented by the same point.
It is not hard to check, however, that this is no restriction 
for our purposes: Our main theorem remains valid if this requirement is dropped.

For two points $x,y\in\mathbb R^2$ we denote by $\dist(x,y)$
the Euclidean distance in the plane. If $X$ is a point set in $\mathbb R^2$
then we let  $\conv(X)$ be the convex hull of the points in $X$. As 
a shorthand we set $\conv(G):=\conv(V(G))$.
We say that a line $L$ \emph{separates} a point $p$ from a point set $X$
if $p$ lies in one of the closed half-planes defined by $L$, while $X$ is contained in 
the other. 

For any set $X$ and any $x$ we use $X+x$ to denote $X\cup\{x\}$.

\section{Proof of main theorem}\label{main}

The proof Theorem~\ref{mainthm} rests on the Edmonds-Gallai decomposition
as well as a key lemma, which will be proved in the course of the following
two sections. 

\begin{lemma}\label{cliques}
Let $G$ be a unit disk graph with $\alpha(G)\leq 2$. 
Then $V(G)$ can be partitioned into three cliques, two of which 
have different cardinalities.
\end{lemma}
Let me remark that the lemma is motivated by the structure of the lower bound example. 
Moreover, as a by-product we obtain that a  stability two unit disk graph has 
clique partition number $\overline\chi\leq 3$.

For the well-known Edmonds-Gallai decomposition, which we state 
below, we refer to Lovasz and Plummer~\cite{LP86}. 
We briefly recall
the basic notions of matching theory.
A \emph{matching} of a graph $G$ is a set of edges so that no two of its edges share 
an endvertex. The matching is \emph{perfect} if every vertex is incident with a matching edge;
and it is \emph{near-perfect} if this is the case for every vertex except one. The graph $G$
is \emph{factor-critical} if $G-v$ has a perfect matching for every vertex $v$.

\begin{theorem}[Edmonds-Gallai decomposition]
For any graph $G$ denote by $A$ the set of those vertices~$v$ for which there exists
a maximum-size matching missing~$v$. Let $X:=(\bigcup_{a\in A}N(a))\sm A$, 
and $B:=V(G)\sm (A\cup X)$.
Then
\begin{itemize}
\item every odd component of $G-X$ is factor-critical and contained in $A$;
\item every even component of $G-X$ has a perfect matching and is contained in $B$; and
\item every non-empty subset $X'$ of $X$ has neighbours in more than $|X'|$ odd components of $G-X$.
\end{itemize}
\end{theorem}

The last conclusion, in particular, means that, by Hall's marriage theorem,
there is a matching in which every edge  has one endvertex in $X$
and the other in an odd component of $G-X$, but in which no odd component 
is incident with two edges of the matching.

\begin{proof}[Proof of Theorem~\ref{mainthm}]
Let $H$ be the complement of $G$.
By the 
Edmonds-Gallai decomposition, we find a matching $M$ of $H$, a vertex set $X$ and 
a set $\mathcal O$ of factor-critical components of $H-X$ so that:
$M_R:=M\cap E(R)$ is a perfect matching for $R:=H-X-\bigcup_{K\in\mathcal O}K$;
$M_K:=M\cap E(K)$ is a near-perfect matching for every $K\in\mathcal O$;
and every $x\in X$ is incident with an edge in $M$ whose other endvertex
lies in some $K\in\mathcal O$. Let us denote the set of matching edges
incident with vertices in $X$ by $M_X$, and let $\mathcal O_X$ 
be the set $K\in\mathcal O$ incident with an edge in $M_X$. Finally,
set $\mathcal O':=\mathcal O\sm\mathcal O_X$. Thus
\[
M=M_R\cup M_X\cup\bigcup_{K\in\mathcal O}M_K,\,
\mathcal O=\mathcal O_X\cup\mathcal O'\text{ and }|M_X|=|\mathcal O_X|.
\]

By Lemma~\ref{cliques}, $R$ can be partitioned into 
three stable sets $A_R,B_R,C_R$, which we may choose so that 
$|A_R|\geq\max(|B_R|,|C_R|)$. Moreover, by the same lemma, every $K\in\mathcal O$
can be partitioned into three stable sets $A_K,B_K,C_K$
so that the three sets do not have the same size. Choosing 
them with  $|A_K|\geq |B_K| \geq |C_K|$ implies
\begin{equation}\label{sizeofC}
|C_K|+1 \leq |A_K|.
\end{equation}
Observe that $A:=A_R\cup\bigcup_{K\in\mathcal O}A_K$
is a stable set of $H$.

Now, we see that
\begin{equation*}
2|M_R|=|V(R)|=|A_R|+|B_R|+|C_R|\leq 3|A_R|,
\end{equation*}
while for every $K\in\mathcal O$ we obtain with~\eqref{sizeofC}
\begin{equation*}
2|M_K|=|V(K)|-1=|A_K|+|B_K|+|C_K|-1\leq 3|A_K|-2.
\end{equation*}
From this it follows that
\begin{eqnarray*}
2|M|&=&2|M_R|+2|M_X|+2\sum_{K\in\mathcal O}|M_K|\\
&\leq&  3|A_R|+ 2|M_X|+\sum_{K\in\mathcal O}(3|A_K|-2)\\
& = & 3|A| + 2|M_X| - 2|\mathcal O| \quad=\quad  3|A| - 2|\mathcal O'|.
\end{eqnarray*}
The matching $M$ together with the set of unmatched vertices, one for each $K\in\mathcal O'$, yields
a clique partition of $H$  of size $|M|+|\mathcal O'|$.
Hence
\[
2\overline \chi(H)\leq 2(|M|+|\mathcal O'|)\leq 3|A| - 2|\mathcal O'| +
2|\mathcal O'| \leq 3\alpha(H). 
\] 
We deduce $\chi(G)\leq\frac{3}{2}\omega(G)$, which finishes the proof.
\end{proof}

Observe that the proof of the theorem does not appeal at all to unit disk graphs. 
In fact, we implicitly show the following result for general graphs:

\begin{lemma}
If the vertex set of a graph $G$ can be partitioned
into $k\geq 2$ cliques not all which have the same cardinality then
$\chi(G)\leq \frac{k}{2}\omega(G)$.
\end{lemma}

\section{Basic geometric facts}\label{geom}

Before beginning with the proof of the key lemma, let me 
collect in this section the basic geometric facts that we will need.

The geometry of \udg s is not linear. For example, the subset  
of points of distance~$\leq 1$ to a given set of vertices,
which is the intersection of several unit disks, can be 
very complex indeed. 
Sometimes this inherent non-linearity can be avoided. That is, instead
of exploiting a concrete realisation of the unit disk graph, 
it is sometimes possible to deduce the desired conclusion only by appealing to 
abstract properties shared by all \udg s. If this is possible 
it might even result in cleaner arguments. The
\textsc{MaxClique} algorithm by Raghavan and Spinrad~\cite{RSp03},
for instance, is such an example. 

Abstract properties of \udg s include the fact that a \udg\ may not contain
any induced $K_{1,6}$ or that 
\begin{equation}\label{nbhprop}
\begin{minipage}[c]{0.8\textwidth}\em
the common neighbourhood of any two 
non-adjacent vertices induces a co-bipartite graph.
\end{minipage}\ignorespacesafterend 
\end{equation}
A similar property may be found in~\cite{CCJ90}.

Although~\eqref{nbhprop} is a fairly powerful property it is even 
in conjunction with stability $\alpha=2$ 
not enough to guarantee a chromatic number of 
$\chi\leq \frac{3}{2}\omega$. 
\newcommand{\cs}[1]{\ensuremath{\text{\rm CS}_#1}}
To see this, consider the following graph \cs k, which is a subgraph of a graph 
appearing in Chudnovsky and Seymour~\cite{ChudSey}.
Let \cs k be defined on four disjoint cliques each of which 
is comprised of $k$ vertices: $\{a_1,\ldots, a_k\}$, $\{b_1,\ldots, b_k\}$,
$\{c_1,\ldots,c_k\}$ and $\{d_1,\ldots,d_k\}$. 
Additionally, for $i,j=1,\ldots,k$ with $i\neq j$ we define the following adjacencies:
Let $a_i$ be adjacent with $b_j$ and $d_j$ and with $c_i$;
let $b_i$ be adjacent with $c_j$ and $a_j$ and with $d_i$;
let $c_i$ be adjacent with $d_j$ and $b_j$ and with $a_i$; and
let $d_i$ be adjacent with $a_j$ and $c_j$ and with $b_i$.
All other pairs of vertices are non-adjacent. 

Clearly, the stability of \cs k is equal to~$2$, and
if $k\geq 3$ then $\omega(\cs k)=k+1$ and $\chi(\cs k)=2k$.
It  is not entirely obvious but also 
not overly difficult to check that \cs k satisfies~\eqref{nbhprop}.

To sum up, directly exploiting the geometry of a \udg\
might be hard due to the inherent non-linearity, while the other approach
of using only abstract properties appears to fail. 
So, what can be done? As always, there is a sensible 
middle ground. We will work with a concrete geometric realisation,
that is, the vertices will have concrete positions in the plane,
but we will in some sense linearise the adjacencies:
To show that two given vertices are adjacent we will 
never try to calculate their distance directly but rather 
use the following two principles,  Lemmas~\ref{crossing} and~\ref{elemtrig},
that are of a more combinatorial flavour.

Let us say that for distinct vertices $u,v,x,y$ in a unit disk graph,
the two edges $uv$ and $xy$ 
are \emph{crossing} if 
$\conv(u,v)$ intersects $\conv(x,y)$.

\begin{lemma}[Breu~\cite{BreuPHD}]\label{crossing}
Let $u,v,x,y$  be four distinct vertices in a unit disk graph~$G$. If $uv$ and $xy$
are crossing edges then $G[u,v,x,y]$ contains a triangle.
\end{lemma}

\begin{lemma}\label{elemtrig}
Let a vertex $v$ of a unit disk graph be adjacent to 
two vertices $u$ and $w$. Then $v$ is adjacent to $x$ for 
every vertex $x\in\conv(u,v,w)$.
\end{lemma}
\begin{proof}
The vertex $v$ is adjacent to every vertex in the unit disk
centered at~$v$. This disk clearly contains $\conv(u,v,w)$.
\end{proof}

Having stated that we will work only with these two principles
rather than with concrete distances, let me turn around and immediately
violate that rule.
This becomes necessary as  we will
need to distinguish two classes of unit disk graphs:
Those that have two vertices that are far apart 
and those that fit into a small disk.
This will be done in the next two lemmas---from then on, however, 
we will adhere to the rule.

\begin{lemma}\label{distordisk}
Let $G$ be a unit disk graph. Then either $G$ has two
vertices of distance greater than $\sqrt 3$ or 
there is a unit disk that contains all of $G$.
\end{lemma}
\begin{proof}
We assume that all pairs of vertices of $G$ have distance at most $\sqrt 3$.
\comment{
We consider the continuous function 
\[
f:\mathbb R^2\to \mathbb R,\quad
x\mapsto \max\{\dist(x,v) : v\in V(G)\},
\]
and claim that 
\begin{equation}\label{globalmin}
\emtext{$f$ has a global minimum}.
\end{equation}
Indeed, choose $m>0$ as the greatest distance between
any two vertices of $G$.  Let $D$
be a closed disk containing $\conv (G)$ so that every point outside $D$ has
at least distance $2m$ to every vertex of $G$. 
As $D$ is compact and $f$ continuous, 
$f$ attains a minimum value~$d$ in $D$. Clearly, $d\leq m$. 
On the other hand, we have $f(p)\geq 2m$ for every point $p$ outside $D$.
Thus $d$ is a global minimum.
}%
It is straightforward to see that there is a point $x$ whose maximal
distance to $V(G)$ is minimal. More formally, there is an $x\in\mathbb R^2$
so that setting $d:=\max\{\dist(x,v) : v\in V(G)\}$
we obtain $\dist(y,v)\geq d$ for all points $y\in\mathbb R^2$ and all vertices $v$.
\comment{
Denote by $x$ a point in $\mathbb  R^2$ that attains the global minimum $d$,
and let $W$ be the set of vertices $v$ with $\dist(x,v)=d$.
}

Let $W$ be the set of vertices $v$ with $\dist(x,v)=d$.
We will see that $d\leq 1$, which means that all of $G$ is contained in 
the unit disk with centre~$x$.
For this, we claim that
\begin{equation}\label{xinside}
x\in\conv (W).
\end{equation}
Suppose, $x\notin\conv(W)$. Then there are two vertices $w,w'$ in $W$, so that the line $L$
through $w$ and $w'$ separates $x$ from $W$, and so that $x\notin L$.  Choose $\epsilon>0$
small enough so that $\dist(v,x)+\epsilon<d$ for all $v\in V(G)\sm W$.
On the line through $x$ that is orthogonal to $L$, let $x'$ be the point between $x$ and $\conv (W)$
of distance~$\epsilon$ to $x$. To see that $f(x')<f(x)$ 
consider any $u\in W$, which then lies
on the segment between $w$ and $w'$ on a circle of radius~$d$ and centre~$x$.
The angle at $x'$ in the triangle with vertices $u,x',x$
is at least $\pi/2$, which means 
that $\dist(u,x')<\dist(u,x)$. On the other hand, for any $v\in V(G)\sm W$ we also have
$\dist(x',v)<d$ by choice of $\epsilon$. Therefore, $x'$ contradicts that $f$ takes a minimum at $x$.

Now, if $x$ is the convex combination of two vertices $w_1$ and $w_2$ in $W$ then
$2d=\dist(w_1,w_2)\leq \sqrt 3$, which implies $d<1$. If this is not the case, 
then, by Carath\'eodory's theorem, $x$ lies in the convex hull of three vertices $w_1,w_2,w_3$
of $W$. There are $i,j$ with $1\leq i<j\leq 3$ so that the angle $\alpha$
at $x$ in the triangle with vertices $x,w_i,w_j$ is
at least $\frac{2}{3}\pi$ but no more than $\pi$ by~\eqref{xinside}. 
Then we get
\[
\frac{\sqrt 3}{2}=\sin\left(\frac{\pi}{3}\right)\leq\sin\left(\frac{\alpha}{2}\right)=\frac{\dist(w_i,w_j)}{2d}
\]
Thus, $d\leq 1$ as $\dist(w_i,w_j)\leq\sqrt 3$.
\end{proof}

The case when  a unit disk graph has two vertices $u,v$
of distance~$\geq\sqrt 3$ is particularly easy. 
We will see that in this case we can get the key lemma
with only a small effort.

\begin{lemma}\label{commonnbh}
Let $G$ be a unit disk graph. If $u$ and $v$ 
are two vertices of distance at least $\sqrt{3}$ then
$N(u)\cap N(v)$ is a clique.
\end{lemma}
\begin{proof}
Consider any two vertices $u$ and $v$ for which
$N(u)\cap N(v)$ contains two non-adjacent vertices.
The two points of greatest distance in the intersection of 
the unit disk centered at $u$ and the unit disk centered at $v$ 
are  the two points where the 
boundaries of the two unit disks meet. Let this distance be 
$s$; and observe that $s>1$ as otherwise $N(u)\cap N(v)$ would be a clique.
Now, we obtain%
\comment{Pythagoras $\rightarrow$ $s^2/4+dist^2/4=1$}
\[
\dist(u,v)=\sqrt{4-s^2}<\sqrt 3.
\]
\end{proof}

\section{The key lemma}\label{keysec}

In this section we will prove a slightly stronger version of the key lemma:
\begin{lemma}\label{xkey}
Let $G$ be a unit disk graph with $\alpha(G)\leq 2$. 
Then $V(G)$ is the union of   three cliques, two of which 
contain a common vertex.
\end{lemma}
This implies Lemma~\ref{cliques}:
Let $A,B,C$ be three cliques whose union is $V(G)$, and that
are disjoint except for a vertex $v$ that is contained in $A$ and $B$
but not in $C$. We assume furthermore that $|A|\geq|B|$. 
Then $A,B\sm\{v\},C$ is a clique partition of $V(G)$ in which not all
of the cliques have the same size.

\medbreak

The key lemma holds trivially if $\conv(G)$ is contained in a line.
We assume from now on that this is not the case.

\newcommand{\hollow}{hollow}
\newcommand{\Hollow}{Hollow}

Let us start by considering a special case.
We say that a unit disk graph $G$ is \emph{\hollow}
if no vertex of $G$ lies in the interior of $\conv(G)$.
Thus, all vertices of $G$ appear on the boundary of the polygon $\conv(G)$.
We fix one of the circular orders, say the clockwise order, in which 
the vertices appear on the boundary. We will use the usual 
interval notation for the vertices of a \hollow\ unit disk graph. 
So, for two distinct vertices $u,v$ 
we denote by $[u,v]$ 
the set of all vertices that appear in clockwise order on the boundary
starting with $u$ up to $v$. If $u=v$, we set $[u,v]=\{u\}$. 
We furthermore define $(u,v]:=[u,v]\sm\{u\}$, $[u,v):=[u,v]\sm\{v\}$
and $(u,v):=[u,v]\sm\{u,v\}$.
We say that $u$ and $v$ are \emph{consecutive} if $u\neq v$ and 
either $(u,v)=\emptyset$ or $(v,u)=\emptyset$.

\Hollow\ unit disk graphs have an advantage over general 
unit disk graphs. To decide whether 
two edges $uv$ and $xy$
cross reduces to determining the order of the endvertices
on the boundary: The edges cross if and only if the endvertices
are interleaved, that is, if and only if both $(u,v)$ and $(v,u)$
meet $\{x,y\}$.

\begin{lemma}\label{3consec}
Let $G$ be a \hollow\ unit disk graph. Assume $G$ to have 
vertices $x_1,y_1,x_2,y_2,x_3,y_3$ appearing in this order on the boundary 
of $\conv(G)$ so that $x_i$ and $y_i$ are non-adjacent for $i=1,2,3$.
Then $\alpha(G)\geq 3$.
\end{lemma}
\begin{proof}
If $x_1$ and $x_2$ fail to be adjacent, set $s=x_2$. If $x_1$ and $x_2$ are adjacent then
$y_1y_2\notin E(G)$; otherwise $x_1x_2$ and $y_1y_2$ would be crossing but $G$ does not
contain any triangle on these four vertices, which is
impossible by Lemma~\ref{crossing}. Setting $s=y_1$, we obtain in both cases
that $s\in (x_1,y_2)$ and that $x_1s$ and $sy_2$ are non-edges.
By symmetry, we find a $t\in (s,y_3)$ so that $t$ is not a neighbour of $s$ nor of $y_3$. 
We consider the four vertices $x_1,s,t,y_3$. Unless $\alpha(G)=2$, we have that $x_1$ is adjacent to $t$
and $y_3$ adjacent to $s$. Then, however, $x_1t$ and $y_3s$ are two crossing edges whose 
endvertices do not induce any triangle, which contradicts Lemma~\ref{crossing}.
\end{proof}

We now prove the key lemma for \hollow\ unit
disk graphs.

\begin{lemma}\label{convcase}
Let $G$ be a \hollow\ unit disk graph with $\alpha(G)\leq 2$,
Then for every vertex $v$
there are vertices $v^+$ and $v^-$ so that $[v^-,v]$, $[v,v^+]$ 
and $V(G)\sm [v^-,v^+]$ are cliques.
\end{lemma}
\begin{proof}
First, we may clearly exclude the case when $G$ is complete.
Let $y^+$ be the last vertex in clockwise direction from $v$ so that 
$[v,y^+)$ forms a clique, that is, we choose $y^+$ so that $[v,y^+)$ is a
clique and maximal among all such cliques.
By choice of $y^+$, there exists a $x^+\in [v,y^+)$ that is non-adjacent to $y^+$.
Similarly, we denote by $y^-$ the last vertex in counterclockwise
direction so that $(y^-,v]$ is a clique, and we let $x^-$ be a non-neighbour
of $y^-$ in $(y^-,v]$. Define $v^+$ to be the clockwise predecessor of $y^+$,
that is, we choose $v^+$ to be the vertex for which $[v,v^+]=[v,y^+)$.

If $y^-\in [v,v^+]$ then put $v^-=y^+$. This ensures that $[v^-,v]$ is a subset
of the clique $(y^-,v]$. As, moreover, $V(G)\sm [v^-,v^+]$ is empty in this case, we are done.
 So, we will assume that $y^-\notin [v,v^+]=[v,y^+)$. Let $v^-$ be the vertex 
for which $[v^-,v]=(y^-,v]$. 
Suppose that $V(G)\sm [v^-,v^+]=[y^+,y^-]$ is not a clique. Thus, there exist non-adjacent 
$r,s\in [y^+,y^-]$, where $s\notin [y^+,r]$. Then $x^+,y^+,r,s,y^-,x^-$
yield three pairs of non-adjacent vertices that are as in  Lemma~\ref{3consec},
which is impossible.
\end{proof}

Next, we will see how the general case can be deduced from the 
\hollow\ case. For this, we first note two simple consequences
 of Lemma~\ref{elemtrig}:

\begin{lemma}\label{inclique}
Let $K$ be a clique in a unit disk graph. 
Then $K+v$ is a clique for any vertex $v\in\conv(K)$.
\end{lemma}
\begin{proof}
Consider any vertex $k\in K$.
Then we see that every point in $\conv(K)$ is contained in a triangle incident with $k$.
Now the assertion follows from Lemma~\ref{elemtrig}.
\end{proof}

\begin{lemma}\label{toclique}
If a vertex $v$ in a unit disk graph is complete to a clique 
$K$ then $v$ is complete to every vertex in $\conv(K)$. 
\end{lemma}
\begin{proof}
As $v$ is complete to $K$, every point of $\conv(K)$ lies in a triangle incident with $v$,
which in turn implies by Lemma~\ref{elemtrig}
that every vertex in $\conv(K)$ is adjacent to~$v$.
\end{proof}

We quickly exclude one more easy case.

\begin{lemma}\label{nonedge}
Let $G$ be a unit disk graph with $\alpha(G)\leq 2$,
and let $u$ and $v$ be two non-adjacent vertices. 
Assume that all vertices 
of $G$ lie on one side of the line through $u$ and $v$, that is, 
the line through $u$ and $v$ does not separate any two points of $\conv(G)$.
 Then 
$N(u)$ is the union of two cliques while $N(v)\sm N(u)$ is a single clique.
\end{lemma}
\begin{proof}
It suffices to show that $N(u)\cap N(v)$ is a clique. 
Consider two common neighbours $x,y$ of $u$ and $v$. 
If $x\in\conv(u,y,v)$ or if $y\in\conv(u,x,v)$ then 
$x$ and $y$ are adjacent by Lemma~\ref{elemtrig}.
So, suppose that neither is the case. 
In a similar way, it follows from $uv\notin E(G)$
that  neither 
$u$ nor $v$ can be contained in the interior of the
convex hull  of the other three. Thus, all 
four vertices lie on the boundary of $\conv(u,v,x,y)$.
Because $\conv(G)$ lies on one side of the line through $u$ and $v$, 
we  deduce that one of the two pairs 
of edges, $uy,vx$ and $ux,vy$, cross. That $x$ and $y$ are adjacent now follows 
from Lemma~\ref{crossing}.
\end{proof}

In our proof of the key lemma there is only one obstacle left, 
which we will overcome with the help of the next lemma. The lemma
is based on the insight that, provided the graph is contained in a 
unit disk, the vertices on the boundary of $\conv (G)$ largely determine
the behaviour of the interior vertices. 
Slightly more precisely, we know from Lemma~\ref{convcase} that the outer
vertices can be partitioned into three cliques, and we will see
that each  interior vertex can easily be assigned to one of these
cliques---with the exception of two small zones of vertices. Handling 
these two zones will be the main difficulty.

\begin{lemma}\label{3cliques}
Let $G$ be a unit disk graph with $\alpha(G)\leq 2$. 
If $G$ is contained  in a unit disk then
 $V(G)$ is the union of three cliques, two of which have a common vertex.
\end{lemma}
\begin{proof}
We will show that there is a vertex $b$ on the boundary of $\conv(G)$, and three cliques 
$B^+,B^-,R$ that cover $V(G)$ so that $b\in B^+\cap B^-$. 

By assumption, there is a point $p\in\mathbb R^2$ of distance at most~$1$
to every vertex in $G$. 
Now, adding $p$ as a vertex to $G$ is entirely harmless: $G+p$ has still 
stability at most two and, assuming we find cliques as desired in $G+p$
with $b\neq p$,
we obtain such cliques for $G$ by simply deleting $p$ from the cliques.
Thus, we may assume that
\begin{equation}\label{universal}
\emtext{
$G$ has a vertex $p$ that is adjacent to every other vertex.
}
\end{equation}

\newcommand{\myb}{\ensuremath{F}}
Denote by $\myb$ the set of vertices of $G$ on the boundary of $\conv(G)$.
Then the graph induced by $\myb$ is a \hollow\ unit disk graph, and we will 
continue to use the interval notation for this induced subgraph of $G$,
and only for this graph. This means that a set 
$[u,v]$ is always understood with respect 
to the \hollow\ unit disk graph on $\myb$, and that in particular $[u,v]\subseteq\myb$.

In light of Lemma~\ref{nonedge} we may assume that 
\begin{equation}\label{bndry}
\emtext{
two consecutive vertices  on the boundary of $\conv(G)$ 
are adjacent.
}
\end{equation}

Pick any vertex $b\in \myb$ other than $p$.
By Lemma~\ref{convcase} we may 
choose $b^+,b^-\in\myb$ so that 
$[b,b^+]$, $[b^-,b]$ and $\myb\sm [b^-,b^+]$ are cliques that meet at most in $b$,
and such that $\myb\sm [b^-,b^+]$ is minimal subject to this condition.
Observe that $b^+=b^-$ is impossible:
This would entail $b=b^+=b^-$ but both 
of 
$[b,b^+]$ and $[b^-,b]$ contain two vertices by~\eqref{bndry}.
Thus, $\myb\sm [b^-,b^+]=(b^+,b^-)$.
If $(b^+,b^-)\neq\emptyset$ let $r^+,r^-$ be so that $[r^+,r^-]=(b^+,b^-)$.
If, on the other hand, $(b^+,b^-)=\emptyset$
then we put $r^+=b^-$ and $r^-=b^+$.

We now have that
\begin{equation}\label{5cliques}
\begin{minipage}[c]{0.8\textwidth}\em
every pair of consecutive vertices in $\myb$ lies in one of
the following cliques:
$[b,b^+]$, $[b^+,r^+]$, $(b^+,b^-)$, $[r^-,b^-]$
and $[b^-,b]$.
\end{minipage}\ignorespacesafterend 
\end{equation} 
If $(b^+,b^-)\neq\emptyset$ then trivially every pair of consecutive vertices
lies in one of the five sets. In the case of $(b^+,b^-)=\emptyset$
we have $[r^-,b^-]=[b^+,b^-]$.

It remains only to verify that $[b^+,r^+]$ and $[r^-,b^-]$ are cliques as claimed.
Indeed, $b^+$ and $r^+$, as well as $b^-$ and $r^-$, are consecutive
vertices on the boundary. Thus, by~\eqref{bndry}, $[b^+,r^+]$ is simply the edge $b^+r^+$,
and $[r^-,b^-]$ coincides with the edge $b^-r^-$. This proves~\eqref{5cliques}.
\medskip

We define
\begin{equation*}
\begin{minipage}[c]{0.8\textwidth}\em
$B^+=\conv([b,b^+]+p)\cap V(G)$, $\,\,B^-=\conv([b^-,b]+p)\cap V(G)$,\\
$T^+=\conv(b^+,r^+,p)\cap V(G)$,
$\quad T^-=\conv(r^-,b^-,p)\cap V(G)$,\\
and $R=\conv([r^+,r^-]+p)\cap V(G)$.
\end{minipage}\ignorespacesafterend 
\end{equation*} 
See Figure~\ref{playground} for an illustration.
Observe that as $p$ is adjacent with every other vertex, it follows that
every vertex is contained in one of the five sets. Moreover, from~\eqref{5cliques}
and Lemma~\ref{inclique}
we deduce that $B^+$, $B^-$, $T^+$, $T^-$ and $R$ are cliques.

\begin{figure}[ht]
\centering
\input{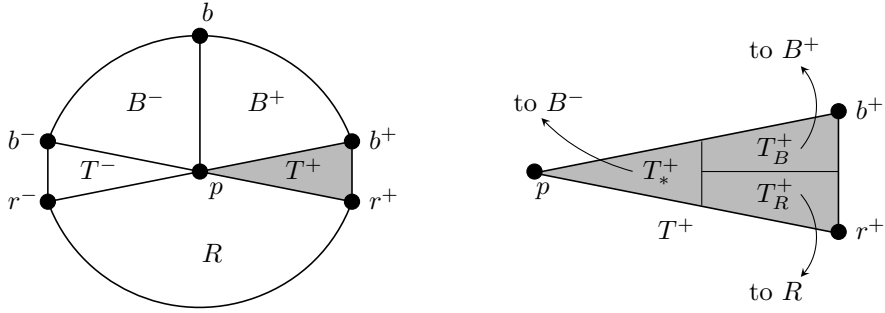}
\caption{The five cliques (left); how to divide up $T^+$ (right)}\label{playground}
\end{figure}

Assume for the moment that $(b^+,b^-)=\emptyset$. Then 
$[b^+,r^+]=[b^+,b^-]=[r^-,b^-]$, which means that $T^+=T^-$.
Thus, we see that $V(G)$ is the union 
of the three cliques $B^+$, $B^-$ and $T^+$, two of which contain $b$.
As we are done in this case, we will assume from now on that
\begin{equation}\label{bs}
\emtext{
$(b^+,b^-)=[r^+,r^-]\neq\emptyset$.
}
\end{equation}

The rest of the proof will be spent on dividing up $T^+\cup T^-$
among the other cliques, so that we obtain three cliques that cover all of $G$.
More precisely, we will partition $T^+$ and $T^-$ into sets $T^+_B,T^+_R,T^+_*$
and $T^-_B,T^-_R,T^-_*$, respectively, so that 
$B^+\cup T^+_B\cup T^-_*$,
$B^-\cup T^-_B\cup T^+_*$ and $R\cup T^+_R\cup T^-_R$ are cliques; see Figure~\ref{playground}. 
As $b$ is contained in the first two of these three, this will complete the proof
of the lemma. In order to do so, we define
\begin{itemize}
\item $T^+_B:=\{t\in T^+:t\emtext{ is complete to }B^+\}$;
\item $T^+_R:=\{t\in T^+\sm T^+_B:t\emtext{ is complete to }R\}$; and
\item $T^+_*:=T^+\sm(T^+_B\cup T^+_R)$.
\end{itemize}
The sets $T^-_B$, $T^-_R$ and $T^-_*$ are defined symmetrically.

We claim that
\begin{equation}\label{TR}
\emtext{
$T^+_R\cup T^-_R$ is  a clique.
}
\end{equation}
Suppose that is not the case. As both $T^+$ and $T^-$ are cliques, 
this means there are non-adjacent $t^+\in T^+_R$ and $t^-\in T^-_R$. 
By definition of $T^+_R$, the vertex  $t^+$ is not complete to $B^+$,
which in turn implies with Lemma~\ref{toclique} that $t^+$ has a 
non-neighbour $s^+$ in $[b,b^+]$. Symmetrically, we find an $s^-\in [b^-,b]$
that is non-adjacent to $t^-$.

We will focus on $t^+,t^-$ and the following  vertices that appear
in this order on the boundary of $\conv (G)$: 
$b^-, s^-, s^+,b^+$. 
Observe that $t^+$ is adjacent to $s^-$ and $t^-$ is adjacent to $s^+$;
otherwise we would obtain a contradiction to $\alpha(G)\leq 2$. Moreover, 
$t^+$ and $b^+$ are adjacent since both are elements of the clique $T^+$. 
In the same way, we have $t^-b^-\in E(G)$. 

Suppose that $t^+b^-\in E(G)$.
Then $t^+$ is adjacent to $p$, to $b^-$ and to $r^-\in R$ (the
last adjacency is because of $t^+\in T^+_R$).
With Lemma~\ref{toclique} we conclude now that 
 $t^+$ is complete to $T^-$,
which is impossible as $t^-\in T^-$. Thus, $t^+$ and $b^-$ are 
non-neighbours, as well as, symmetrically, $t^-$ and $b^+$.

Next, let us note that $s^-$, $s^+$, $b^+$, $b^-$ are pairwise distinct. 
Indeed, the fact that $t^+$ is adjacent to $b^+$ and $s^-$
but not to $s^+$ nor to $b^-$ implies that 
 $s^+\neq s^-$, $s^+\neq b^+$, 
$s^-\neq b^-$  and $b^+\neq b^-$. All other identities
are excluded by the fact that $s^-$, $s^+$, $b^+$, $b^-$ appear in 
this order on $\myb$.

To conclude, we find two disjoint paths $s^+t^-b^-$ and $s^-t^+b^+$
with interleaved endvertices: both $(s^+,b^-)$ and $(b^-,s^+)$
meet $\{s^-,b^+\}$. 
Thus an edge of the first path needs
to cross an edge of the second path. However, one may easily check 
with Lemma~\ref{crossing} that 
none of the following pairs of edges may cross:
$s^-t^+$ and $s^+t^-$;\,
$b^+t^+$ and $b^-t^-$;\,  
$s^+t^-$ and $b^+t^+$;\, $s^-t^+$ and $b^-t^-$.
This finishes the proof of~\eqref{TR}.
\medskip

We show next that
\begin{equation}\label{T*}
\emtext{
$T^+_*$ is complete to $B^-\cup T^-$, and $T^-_*$ is complete to $B^+\cup T^+$.
}
\end{equation}
By symmetry it suffices to show that any $t\in T^+_*$ is complete to $B^-\cup T^-$.
We first observe that 
\begin{equation}\label{xy}
\begin{minipage}[c]{0.8\textwidth}\em
there are distinct
 non-neighbours $x,y$ of $t$ so that $[x,y]$ is a clique, $t$ is 
complete to $(y,x)$ and $t\notin \conv([x,y])$.
\end{minipage}\ignorespacesafterend 
\end{equation} 
To see~\eqref{xy}, denote by $X$ the set of non-neighbours of $t$ in 
$\myb$. Note that $t$ has by definition of $T^+_*$
one non-neighbour in $B^+$ and one in $R^+$. Lemma~\ref{toclique}
implies that $t$ has therefore a non-neighbour in $[b,b^+]$  and in $[r^+,r^-]$,
and consequently that $|X|\geq 2$.
Since $\alpha(G)\leq 2$, the set  $X$ has to be a clique. We deduce from 
Lemma~\ref{inclique} that $t\notin\conv(X)$. 
Thus, there are two vertices $x,y\in X$ so that
the line through $x$ and $y$ separates 
$t$ from $\conv(X)$. Moreover, $t$ does not lie on the line.
 We choose $x$ and $y$ so that $X\subseteq [x,y]$. 

Now consider the unit disk graph on $[x,y]+t$, which is a \hollow\ unit disk graph.
Since $y$ is non-adjacent 
to $t$, which in turn is non-adjacent to $x$, we can employ Lemma~\ref{3consec}
to deduce that $[x,y]$ is a clique. This finishes~\eqref{xy}.
\medskip

We  distinguish four cases.

\case\label{c1} $x\in [r^+,b]$ and $y\in [r^+,b)$.

Since $(b,b^+]$ is disjoint from $\{x,y\}$, we have that 
$(b,b^+]\subseteq (x,y)$ or $(b,b^+]\subseteq (y,x)$.
Then $[b,b^+]\subseteq [x,y)$ or $[b,b^+]\subseteq (y,x)$. 
By definition of $T^+_*$, $t$ has a non-neighbour in $B^+$,
and thus also in $[b,b^+]$ by Lemma~\ref{toclique}.
As 
$t$ is complete to $(y,x)$ by~\eqref{T*},
it follows that $[b,b^+]\subseteq [x,y)$.
As a subset of  $[x,y]$, the set 
$[b,r^+]$ is a clique that strictly contains $[b,b^+]$.
By~\eqref{bs} this implies that $\myb\sm [b^-,r^+]$ is strictly
smaller than $(b^+,b^-)=\myb\sm [b^-,b^+]$, which contradicts
the choice of $b^+$ and $b^-$. 

\case\label{c2} $\{x,y\}\subseteq [b^-,b^+]$.

As $t$ is complete to $(y,x)$ but has by definition of $T^+_*$ 
a non-neighbour in $R$ and thus in $[r^+,r^-]$ it follows that 
$[r^+,r^-]\subseteq (x,y)$. Thus, $[b^+,b^-]\subseteq [x,y]$ is a clique
by~\eqref{xy}, which implies with Lemma~\ref{inclique}
that $T^+\cup R\cup T^-$ is a clique. This, however, is impossible as $t\in T^+_*$
is supposed to have a non-neighbour in~$R$. 

\case\label{c3} $x\in [r^+,r^-]$ and $y\in [b,b^+]$.

In this case, we find that $[b^-,b]\subseteq (x,y]$, and thus that
$[r^-,b]\subseteq [x,y]$ is a clique that strictly contains $[b^-,b]$,
which is impossible by~\eqref{bs} and the choice of~$b^+$ and~$b^-$.

\case\label{c4} $x\in (b,b^+]$ and $y\in [r^+,r^-]$.

As $[b^-,b]\subseteq (y,x)$, it follows from~\eqref{xy} that $t$ is complete
to $[b^-,b]$ and thus to $B^-$ by Lemma~\ref{toclique}. Next, let us show 
that $t$ is complete to $T^-$ as well. If $r^-=y$ then all of $[r^+,r^-]$
is contained in the clique $[x,y]$ as well as $b^+$. From Lemma~\ref{inclique}
we deduce that $R\cup T^+$ is a clique, which is impossible as $t\in T^+_*\subseteq T^+$
cannot be complete to $R$.
Thus, $r^-\in (y,x)$, which means that $t$ is adjacent to $r^-$. 
As $t$ is also adjacent to $b^-\in B^-$ it follows from Lemma~\ref{toclique}
that $t$ is complete to $T^-$, as desired.

\medskip
The Cases~1--\arabic{ccount} cover all possible values for $x$ and $y$.
Indeed, assume first that neither $x$ nor $y$ is equal to $b$. Then 
any pair of $x,y$ not treated in Case~\ref{c1} 
either lies completely in $(b,b^+]$, or one of $x,y$ lies in $(b,b^+]$
and the other in $[r^+,b)$. However, the former configuration is covered by
Case~\ref{c2}; of the latter it remains to consider the case when
one of $x,y$ lies in $(b,b^+]$ and the other in $[r^+,r^-]$. This is dealt with 
in Cases~\ref{c3} and~\ref{c4}.
 So, assume now that $x=b$. 
Then $y\in[r^+,b)$ falls under Case~\ref{c1}, and $y\in (b,r^+)=(b,b^+]$
under Case~\ref{c2}. Finally, if $y=b$ then $x\in[b^-,b^+]$ is covered
by Case~\ref{c2}, while $x\in [r^+,r^-]$ is covered by Case~\ref{c3}.

\medskip 

We therefore have proved~\eqref{T*}.
Now the definitions of $T^+_B,T^+_R,T^+_*$
and $T^-_B,T^-_R,T^-_*$ together with~\eqref{TR} and~\eqref{T*} 
imply directly that $B^+\cup T^+_B\cup T^-_*$,
$B^-\cup T^-_B\cup T^+_*$ and $R\cup T^+_R\cup T^-_R$ are cliques. \sloppy
\end{proof}

We can finally prove our key lemma:

\begin{proof}[Proof of Lemma~\ref{xkey}]
We need to find three cliques whose union is $V(G)$ so that 
two of them share a vertex.
Assume first that there are two vertices $u,v$ of distance $\geq\sqrt 3$.
Because of $\alpha(G)\leq 2$, we see that $(N(u)\sm N(v))+u$ 
as well as $N(v)\sm N(u)$ induce cliques. 
These two together with $(N(u)\cap N(v))+u$, which is a clique by Lemma~\ref{commonnbh},
form three cliques as stated.
If all pairs of vertices have distance at most $\sqrt 3$ then the 
assertion follows directly from Lemmas~\ref{distordisk}
and~\ref{3cliques}.
\end{proof}

\medbreak

\noindent{\bf Acknowledgment}\\[3pt]
I am very grateful for inspiring discussions with Ross Kang, as well
as with Naho Fujimoto.

\bibliographystyle{amsplain}
\bibliography{../graphs}

\small
\vskip2mm plus 1fill
\parindent=0pt


$\,$\vfill
Version 30 Sept 2011
\bigbreak

Henning Bruhn
{\tt <bruhn@math.jussieu.fr>}\\
\'Equipe Combinatoire et Optimisation\\
Universit\'e Pierre et Marie Curie\\
4 place Jussieu\\
75252 Paris cedex 05\\
France

\end{document}